\documentclass[11pt,amsmath,amsthm,amsfonts,amssymb,amscd]{article}
\usepackage{epsfig}
\usepackage[latin1]{inputenc}
\usepackage{amsthm}
\usepackage{amsmath}
\usepackage{graphicx}
\font \Bbbten=msbm10 \font \Bbbsev=msbm7 \font \Bbbfiv=msbm5
\newfam\Bbbfam \textfont\Bbbfam = \Bbbten
\scriptfont\Bbbfam=\Bbbsev \scriptscriptfont\Bbbfam=\Bbbfiv
\newcommand{\N}{\mbox{$I\!\!N$}}
\newcommand{\Z}{\mbox{$Z\!\!\!Z$}}

\newcommand{\R}{\mbox{$I\!\!R$}}

\newcommand{\dist}{{\rm dist}}

\newcommand{\dif}{{\rm Diff}}
\newcommand{\diam}{{\rm diam} }

\newcommand{\cita}[7]{{\sc #1, }{\it #2, }{\small #3, {\bf #4 } (#5), p.
#6-#7.}}
\newcommand{\cit}[5]{{\sc #1, }{\it #2, }{\small #3, {\bf #4 } #5.}}
\def\cC{{\mathcal C}}
\newtheorem{maintheorem}{Theorem}

\newtheorem{Df}{Definition}[section]
\newtheorem{Teo}{Theorem}[section]
\newtheorem{Lem}[Teo]{Lemma}

\newtheorem{Prop}[Teo]{Proposition}
\newtheorem{Obs}[Teo]{Remark}

\title{On measure expansive diffeomorphisms.}
\author{
M. J. Pacifico\footnote{partially supported by CNPq Brazil, Pronex on Dynamical Systems,
FAPERJ, Balzan Research Project of J.Palis},
J. L. Vieitez\footnote{partially supported by Grupo de Investigaci\'on "Sistemas Din\'amicos" CSIC (Universidad de la
Rep\'ublica), SNI-ANII, PEDECIBA, Uruguay}}

\date{\today}

\begin{document}
\maketitle
\begin{abstract}
Let $f: M \to M$ be a diffeomorphism defined on a
compact boundaryless $d$-dimensional manifold $M$, $d\geq 2$. C. Morales has proposed the notion of measure expansiveness. In this note we show that diffeomorphisms in a residual subset far from homoclinic tangencies are measure expansive. We also show that surface diffeomorphisms presenting  homoclinic tangencies can be $C^1$-approximated by non-measure expansive diffeomorphisms.
\end{abstract}


%

\tableofcontents

\section{Introduction}
The notion of expansiveness was introduced by Utz in the middle of the twentieth century, see \cite{Ut}.
Roughly speaking a system is expansive if two orbits cannot
remain close to each other under the action of the system.
This notion is very important in the context of the theory of Dynamical Systems.
For instance, it is responsible for  many chaotic properties for homeomorphisms defined on compact spaces, see for instance \cite{Hi}, \cite{Le}, \cite{Ft}, \cite{Vi} for more on this.
There is an extensive literature concerning expansive systems and a classical result
establishes that  every hyperbolic $f$-invariant subset $\Lambda\subset M$ is expansive.

As pointed out by Morales \cite{Mo}, in light of  the rich consequences of expansiveness
in the dynamics of a system,
it is natural to consider another notions of expansiveness. In this same paper
he  introduced a notion generalizing the usual concept of expansiveness.

In this paper we prove that there is a residual subset $\mathcal{G}$ of $\dif^1(M)\backslash \overline{\{\mathcal{H}\mathcal{T}\}}$
such that if $f\in\mathcal{G}$ then $f$ is $\mu$-expansive (see Definition \ref{mu-expansivo}). Here $\mathcal{H}\mathcal{T}$ is the subset of $\dif^1(M)$ presenting a homoclinic tangency (see Definition \ref{homtang}).

Moreover we also show that surface diffeomorphisms presenting  homoclinic tangencies associated to hyperbolic periodic points can be $C^1$-approximated by non measure-expansive diffeomorphisms.

\section{Preliminary results and statement of the main result}

Let us start with the different definitions of expansiveness we shall  deal with.
To this end we define for $x\in X$, where $(X,d)$ is a compact metric space, the set
\begin{equation} \label{Gama}
\Gamma_\epsilon(x,f)\equiv \{y\in X\,/\,
d(f^n(x),f^n(y))\leq \epsilon,\, n\in\Z\}\, .
\end{equation}

 We simply
write $\Gamma_\epsilon(x)$ instead of $\Gamma_\epsilon(x,f)$ when
it is understood  which $f$ we refer to.


\subsection{Expansiveness and robust expansiveness.}

\begin{Df} \label{expansivo}
Let $f:X\to X$ be a homeomorphism defined on a compact metric space $(X,d)$.
 We say that $f$ is an expansive homeomorphism if there is $\alpha>0$ such that $\Gamma_\alpha(x)=\{x\}$ for all $x\in X$.
 Equivalently, given $x,y\in X$, $x\neq y$, there is $n\in\Z$ such that $\dist(f^n(x),f^n(y))>\alpha$.
\end{Df}

For $f$  a diffeomorphism one is interested in the relation between a given property
in the underlying dynamics and its influence on the dynamics on the infinitesimal level,
i. e., in the dynamics of the tangent map $Df: TM\to TM$.
Usually one cannot expect that a sole notion on the underlying dynamics can guarantee
any interesting feature on the infinitesimal level.
Hence we ask for a robust property valid in a whole neighborhood of $f\in \dif^r(M)$, $r\geq 1$.

\begin{Df} \label{robustinha}
A compact $f$-invariant subset  $\Lambda$ is $C^r$-robustly expansive, $r\geq 1$, if and only if there exists
a $C^r$-neighbourhood ${\cal U}(f)$ of $f$ such that for all
$g\in {\cal U}(f)$, there exists a continuation of $\Lambda_g$,
such that $g|_{\Lambda_g}$ is expansive.
\end{Df}

We prove at \cite{PPV,PPSV, SV} that when $\Lambda=H(p,f)$ is a robustly $C^1$-expansive homoclinic class associated to a hyperbolic periodic point $p$ then $H(p,f)$ is hyperbolic (see Subsection \ref{part-hip}).

\medbreak

\subsection{Entropy expansiveness and robust entropy expansiveness.}

Another notion of expansiveness introduced by Bowen at \cite{Bo} is that of an entropy expansive homeomorphism $f:M\to M$, or $h$-expansive homemorphism for short.

Let $K$ be a compact invariant subset of $M$ and $\dist:M\times
M\to\R^+$ a distance in $M$ compatible with its Riemannian
structure. For $E,F \subset K$, $n\in\N$ and $\delta>0$ we say that
$E$ $(n,\delta)$-spans $F$ with respect to $f$ if for each $y\in F$
there is $x\in E$ such that $\dist(f^j(x),f^j(y))\leq \delta$ for
all $j=0,\ldots ,n-1$. Let $r_n(\delta,F)$ denote the minimum
cardinality of a set that $(n,\delta)$-spans $F$. Since $K$ is
compact $r_n(\delta,F)<\infty$. We define
$$h(f,F,\delta)\equiv \lim\sup_{n\to\infty}\frac{1}{n}\log(r_n(\delta,F))$$
and the topological entropy of $f$ restricted to $F$ as
$$h(f,F)\equiv \lim_{\delta\to 0}h(f,F,\delta)\, .$$
The last limit exists since $h(f,F,\delta)$ increases as
$\delta$ decreases to zero.

\begin{Df} \label{gammadeepsilon}
We say that $f/K$ is {\em
entropy-expansive} or {\em $h$-expansive} for short, if and only if
there exists $\epsilon>0$ such that
$$h^*_f(\epsilon)\equiv \sup_{x\in K}h(f,\Gamma_\epsilon(x)) = 0\, .$$
\end{Df}

As for the case of expansiveness we may define a notion of robust $h$-expansiveness.

\begin{Df} \label{robuh}
If $f:M\to M$ is a $C^r$-diffeomorphism , $r\geq 1$, and $K\subset
M$ is compact invariant, we say that $f/K$ is robustly $C^1$-entropy
expansive if there is a $C^1$-neighborhood $\mathcal{U}$ of $f$ and
an open set $U\supset K$ such that if $g\in \mathcal{U}$ then there
is $K_g\subset U$ such that $g/K_g$ is entropy expansive. We say
that $K_g$ is a continuation of $K$ (not necessarily unique).
\end{Df}

%

We prove at \cite{PaVi,PaVi2,DFPV} that if $K$ is a homoclinic class  $H(p,f)$ associated to a hyperbolic periodic point $p$ then it is robustly $h$-expansive if and only if it admits a finest dominated splitting
$$T_{H(p)}M=E\oplus F_1\oplus\cdots \oplus F_k\oplus G$$
with $F_j$ one dimensional sub-bundles, $E$ uniformly contracting and $G$ uniformly expanding.

Other class of robust entropy expansive diffeomorphims is that of Morse-Smale diffeomorphisms. Indeed, all of them have topological entropy zero in a robust way.

\subsection{Domination, partial hyperbolicity, hyperbolicity.} \label{part-hip}
Recall the notion of a dominated splitting for a compact $f$-invariant subset $\Lambda\subset M$ of a diffeomorphism $f:M\to M$\,. It can be seen as a weak form of hyperbolicity.

\begin{Df}
\label{domi}
We say that a compact $f$-invariant set $\Lambda\subset M$ admits a
dominated splitting if the tangent bundle $T_{\Lambda}M$ has a
continuous $Df$-invariant splitting $E\oplus F$ and there exist
$C>0,\, 0<\lambda <1$, such that

\begin{equation*} \label{domino}
\|Df^n|E(x)\|\cdot\|Df^{-n}|F(f^n(x))\|\leq C\lambda^n
\;\forall x\in \Lambda,\, n\geq 0.
\end{equation*}

When the dominated splitting can be written as a sum
\begin{equation} \label{deco}
T_\Lambda M=E_1\oplus\cdots \oplus E_j\oplus E_{j+1}\oplus\cdots\oplus E_k
\end{equation}
we say that this sum is dominated if for all $j$ the sum
$$(E_1\oplus\cdots \oplus E_j)\oplus (E_{j+1}\oplus\cdots\oplus E_k)$$
is dominated.
\end{Df}

 If we cannot decompose in a non-trivial way any sub-bundle $E_j$ appearing at equation (\ref{deco}) we say that it is the {\em  finest} dominated splitting.



  Next we define partial hyperbolicity and hyperbolicity.
 \begin{Df}
 We say that a compact $f$-invariant set $\Lambda\subset M$ is partially hyperbolic
 if the tangent bundle $T_{\Lambda}M$ has a dominated splitting
 $E^s\oplus F\oplus E^u$ and there exist
$C>0,\, 0<\lambda <1$, such that for all vectors $v\in E^s$ we have $\|Df^n(v)\|\leq C\lambda^n \|v\|$ for all $n\geq 0$ and
 for all vectors $v\in E^u$ we have $\|Df^{-n}(v)\|\leq C\lambda^n \|v\|$ for all $n\geq 0$. Vectors in $F$ are less expanded
 than vectors in $E^u$ and less contracted than vectors in $E^s$ (this follows from domination).
\end{Df}

\begin{Obs}
 In case that the central sub-bundle $F$ is trivial, we say that $\Lambda$ is {\em hyperbolic}.
\end{Obs}


\subsection{Measure expansiveness.}


 Next we introduce the notion of measure expansiveness given by Morales.

\begin{Df}[see \cite{Mo}] \label{mu-expansivo}
Let $f:X\to X$ be a homeomorphism defined on a compact metric space $(X,d)$ and $\mu$ a non-atomic
probability measure defined on $X$ (not necessarily $f$-invariant). We say that $f$ is a $\mu$-expansive homeomorphism if there is $\alpha>0$ such that $\mu(\Gamma_\alpha(x))= 0$ for all $x\in X$. Here $\Gamma_\alpha(x)$ is the set defined at equation (\ref{Gama}).
\end{Df}

We will show that $C^1$ generically diffeomorphisms far away from homoclinic tangencies are measure expansive. To that end we recall the definition of homoclinic tangencies.

\begin{Df} \label{homtang}
 A diffeomorphism $f : M \to M$ exhibits a homoclinic tangency if
there is a hyperbolic periodic orbit $\mathcal{O}$ whose invariant manifolds $W^s (\mathcal{O})$ and $W^u (\mathcal{O})$
have a non transverse intersection.
\end{Df}

We set $\mathcal{H}\mathcal{T}$ for the subset of $\dif^1{M}$ constituted of
diffeomorphisms presenting a homoclinic tangency.
Given a subset $A$ of $\dif^1{M}$ we use the notation $\overline{A}$ for the closure of $A$
in $\dif^1{M}$.
\vspace{0.2cm}

The main results in this paper are the following theorems:

\begin{maintheorem} \label{Teo C}
Let $f:M\to M$ be a $C^1$-diffeomorphism defined on a compact manifold $M$.
There is a $\mathcal{G}$ residual subset of $\dif^1{M}\backslash \overline{\mathcal{H}\mathcal{T}}$
such that for any Borel probability measure $\mu$ (invariant by $f$ or not)
we have that there is $\delta>0$ such that $\mu(\Gamma_\delta(x))=0$ for all $x\in M$.
In particular $f$ is $\mu$-expansive.
\end{maintheorem}

\begin{maintheorem} \label{Teo D}
Let $f:M\to M$ be a $C^1$-diffeomorphism defined on a
compact surface $M$ having a homoclinic tangency associated to a hyperbolic periodic orbit $\mathcal{O}$.
Then there is an arbitrarily small $C^1$-perturbation of $f$ giving a diffeomorphism $F:M\to M$ which is not measure-expansive.
\end{maintheorem}

\section{Proof of Theorem \ref{Teo C}.}

We start stating some results proved elsewhere that will be used in the proof.
\medbreak

Let $X= \dif^1(M)\backslash \overline{\mathcal{H}\mathcal{T}}$.
The following result is Theorem 1.1 of \cite{CSY}
\begin{Teo} \label{CSY}
The diffeomorphisms $f$ in a dense $\mathcal{G}_\delta$ subset $\mathcal{G} \subset Diff^1(M) \backslash \overline{\mathcal{H}\mathcal{T}}$
has the following properties.
\begin{enumerate}
\item
 Any aperiodic class $\mathcal{C}$ is partially hyperbolic with a one-dimensional cen-
tral bundle. Moreover, the Lyapunov exponent along $E^c$ of any invariant
measure supported on $\mathcal{C}$ is zero.
\item
 Any homoclinic class $H(p)$ has a partially hyperbolic structure
$$T_{\mathcal{C}}M = E^s \oplus E^c_1\oplus\cdots\oplus E^c_k\oplus E^u\ ,.$$
Moreover the minimal stable dimension of the periodic orbits of $H(p)$ is
$dim(E^s)$ or $dim(E^s) + 1$. Similarly the maximal stable dimension of the
periodic orbits of H(p) is $dim(E^s) + k$ or $dim(E^s) + k - 1$. For every
$i$, $1 \leq i \leq k$ there exist periodic points in $H(p)$ whose Lyapunov exponent
along $E^c_i$ is arbitrarily close to 0.
 In particular if $f\in\mathcal{G}$ then $f$ is partially hyperbolic.
\end{enumerate}
\end{Teo}

For $x\in\Lambda$ and $i\in\{1,..., k\}$ let us denote
\begin{equation}\label{e.pp}
\begin{array}{rlll}
 E^{cs,i}(x) :=E^s(x)\oplus E^c_1(x)\oplus \cdots\oplus E^c_i(x); \,\,
 E^{cu,i}(x) :=E^c_{i}(x)\oplus\cdots \oplus E^c_k(x)\oplus E^u(x).
  \end{array}
\end{equation}
We also let
$E^{cs,0}=E^s$ and $E^{cu,k+1}=E^u$ and write $s=\dim (E^s)$ and $u=\dim (E^u)$.
\medbreak


Let us recall the properties of fake central manifolds $\widehat W^{cs} $ due to Burns and Wilkinson, \cite{BW}, see also \cite{DFPV}.
\begin{Prop}\label{p.fake}
Let $f:M\rightarrow M$ be a $C^1$ diffeomorphism and $\Lambda$ a
compact
$f$-invariant set with a partially hyperbolic splitting,
$$T_{\Lambda}M=E^s\oplus E^c_{1}\oplus \cdots \oplus E^c_{k}\oplus E^u.$$
Let $E^{cs,i}$ and $E^{cu,i}$ be as in equation \eqref{e.pp} and
consider their extensions
$\tilde E^{cs,i}$ and $\tilde E^{cu,i}$ to a small neighborhood
of $\Lambda$.

Then for any $\epsilon>0$ there exist constants $R>r>r_1>0$ such that, for every $p\in \Lambda$,
the neighborhood $B(p,r)$ is foliated by foliations $\widehat{W}^u(p)$, $\widehat{W}^s(p)$,
 $\widehat{W}^{cs,i}(p)$, and $\widehat{W}^{cu, i}(p)$, $i\in\{1,..., k\}$,
 such that for each $\beta\in \{u,s, (cs, i), (cu,i)\}$
the following properties hold:
\begin{enumerate}
\item[(i)] {\em{Almost tangency of the invariant distributions.}}  For each $q\in B(p,r)$,
the leaf $\widehat{W}^{\beta}_p (q)$ is $C^1$, and the tangent space $T_q\widehat{W}^{\beta}_p(q)$
lies in a cone of radius $\epsilon$ about $\widetilde{E}^{\beta}(q)$.
\item[(ii)] {\em{Coherence.}}  $\widehat{W}^s_p$ subfoliates $\widehat{W}^{cs,i}_p$  and
$\widehat{W}^u_p$ subfoliates $\widehat{W}^{cu, i}_p$ for each $i\in\{1,..., k\}$.
\item[(iii)] {\em{Local invariance.}}  For each $q\in B(p, r_1)$ we have
$$
f(\widehat{W}^{\beta}_p(q,r_1))\subset \widehat{W}^{\beta}_{f(p)}(f(q))\textrm{ and }
f^{-1}(\widehat{W}^{\beta}_p(q,r_1))\subset \widehat{W}^{\beta}_{f^{-1}(p)}(f^{-1}(q)),
$$
here $\widehat{W}^{\beta}_p(q,r_1)$ is the connected component of
$\widehat{W}^{\beta}_p(q)\cap B(q,r_1)$ containing $q$.
\item[(iv)] {\em{Uniquencess.}}  $\widehat{W}^s_p(p)=W^s(p,r)$ and $\widehat{W}^u_p(p)=W^u(p,r)$.
\end{enumerate}
\end{Prop}
\begin{proof}
See \cite[Section 3]{BW}.
\end{proof}

Given $j\in \{1,\dots,k\}$,
using Proposition \ref{p.fake},
we consider a small $r$ and the
submanifold
\begin{equation}\label{e.ffake}
\widetilde W^{cs,j}(x)=\bigcup_{z\in \gamma_j(x)}\,
\widehat W^{cs,j-1}_x(z,r).
\end{equation}
This submanifold has dimension $s+j$ and is transverse
to $\widehat W^{cu,j+1}_x(z)$ for all $z$ close to $x$.
Note that
$\widetilde W^{cs,1}(x)$ is foliated by stable manifolds
(recall that $\widehat W^{cs,0}_x(z)\subset W^{s}(z)$).

The  next two lemmas follow straightforwardly
from the fact that the angles
between unitary vectors in the cone fields
$\cC (E^{cs,j})$ and
 $\cC(E^{cu,j+1})$ are uniformly bounded away from zero.

\begin{Lem}\label{l.fakeanglesa}
 There is $\kappa>0$ such that for every $j\in \{1,\dots ,k\}$ and
every $\delta>0$ small enough the following property holds:

For every $x\in \Lambda$, every $y\in B_\delta(x)$,
every local submanifolds $N(x)$ of dimension  $s+j$
tangent to the conefield $\cC (E^{cs,j})$ containing $x$
and $M(y)$ of dimension $(k-j)+u$
tangent to the conefield  $\cC(E^{cu,j+1})$ containing $y$
one has that
$N(x) \cap M(y)$ is contained $B_{\kappa\,\delta}(x)$.
\end{Lem}

\begin{Lem}\label{l.fakeanglesb}
 There is $\kappa>0$ such that for every $j\in \{1,\dots ,k\}$ and
every $\delta>0$ small enough the following property holds:

Take any  $x\in \Lambda$
and the local manifold $\widetilde W^{cs,j}(x)$ in \eqref{e.ffake}.
For every $y\in B_\delta(x)\cap \widetilde W^{cs,j}(x)$
one has that
$\gamma_j(x) \cap
\widehat W^{cs,j-1}_x (y)$ is contained in $B_{\kappa\,\delta}(x)$.
\end{Lem}

 As a consequence of Theorem \ref{CSY} and Lemmas \ref{l.fakeanglesa} and \ref{l.fakeanglesb} we have
 \begin{Teo} \label{t.med0}
 Let $\mu$ be a Borel probability measure defined on $M$ and let $f\in\mathcal{G}$ where $\mathcal{G}$ is as in Theorem \ref{CSY}.
 Then there is $\delta>0$ such that $\mu(\Gamma_\delta(x))=0$ for all $x\in M$.
 \end{Teo}
 \begin{proof}
Let $\mu$ be a Borel probability measure of $M$ and choose $x\in \Omega(f)$. Then there is either an aperiodic class or a homoclinic class $H$ in $\omega(x)$. In any case $H$ is partially hyperbolic, since we are assuming that $f\in\mathcal{G}$.

Let $T_HM=E^s \oplus E^c_1\oplus\cdots\oplus E^c_k\oplus E^u\ ,$ with $E^s$ uniformly contracting and $E^u$ uniformly expanding.  Assume $\theta>1$ is the minimum rate of expansion of $E^u$ for $z\in\Lambda$.
Let $c>0$ such that $(1-c)\theta>1$ and find $\delta>0$ less or equal than that of Lemmas \ref{l.fakeanglesa} and \ref{l.fakeanglesb} and also less than $r/2$ where $r>0$ is given by Proposition \ref{p.fake}, such that if
$$\dist(x,y)\leq (\kappa+1)\delta\quad \mbox{then}\quad 1-c\leq\frac{\|Df|_{E^u(y)}\|}{\|Df|_{E^u(x)}\|}\leq 1+c\,.$$
 For this choice of $\delta$ it holds that $\mu(\Gamma_\delta(x))=0$. For, if $y\in\Gamma_\delta(x)$ then letting $y_u$ be the projection of $y$ into $\widehat W^u(x)$ along $\widetilde W^{cs,k}(y)$ if it were the case that $\dist(x,y_u)>0$ then setting $\theta'=(1-c)\theta$ then we get for $n\geq 1$
$$\dist(f^n(x),f^n(y_u))\geq (\theta')^n\dist(x,y_u)\, .$$
Since $\theta'>1$ eventually $\dist(f^n(x),f^n(y_u)>\kappa\, \delta$ and hence, by Lemmas \ref{l.fakeanglesa} and \ref{l.fakeanglesb} we obtain $\dist(f^n(x),f^n(y)>\delta$ contradicting the fact that $y\in\Gamma_\delta(x)$.

Thus $\Gamma_\delta(x)$ is contained in $\widehat W^{cs,k}(x)$. By backward iteration we also get that
$\Gamma_\delta(x)\subset \widehat W^{cu,0}(x)$.

This implies that $\mu(\Gamma_\delta(x))=0$ finishing the proof for $x\in\Omega(x)$.

Assume now that $x$ is any point in $M$. Then by forward iteration we find $N>0$ such that
$f^n(x)\in B(\omega(x),r/4)$ where $r>0$ is as in Proposition \ref{p.fake}. Since $\omega(x)\subset \Omega(f)$ we have that there is either an aperiodic class or a homoclinic class $H$ which is partially hyperbolic such that $\omega(x)\subset H$ and therefore $f^n(x)\in B(H,r/4)$ for $n\geq N$. The result follows as in the case $x\in\Omega(f)$. Similarly by backward iteration we find $N'>0$ such that $f^{-n}(x)\in B(\alpha(x),r/4)$ for $n\geq N'$.

We may conclude, using similar estimations as in the case $x\in\Omega(f)$, that $\mu(\Gamma_\delta(x))=0$.

 \end{proof}

Theorem \ref{Teo C} is an immediate consequence of
Theorem \ref{t.med0} .

 \begin{Obs}
 For the validity of Theorem \ref{t.med0} it is enough to have that $E^s$ is uniformly contracting for the $\alpha$- limit of $x$ or $E^u$ is uniformly expanding for the $\omega$-limit of $x$.
 \end{Obs}

\section{Surface diffeomorphisms in ${\mathcal{H}\mathcal{T}}$.}

In the remaining of the paper $M$ is a compact boundaryless surface.

Let $f:M\to M$ be a diffeomorphism and assume that $f$ exibits a homoclinic
tangency associated to a hyperbolic periodic point $p$ of $f$.

\subsection{ Horseshoes with positive Lebesgue measure.}
%
%
%

It is proved at \cite{Bo2} that there is a $C^1$ horseshoe with positive Lebesgue measure.
In \cite{RY} it is constructed a such a horseshoe 
fattening up an
invariant horseshoe $\Lambda$ to have positive Lebesgue measure as Bowen did. They obtain this fatted horseshoe modifying a diffeomorphism $f$ defined in a square $B=[0,1]\times[0,1]$ so that $f|B$ gives a linear evenly spaced full shift on 2 symbols, see \cite[\S 1]{RY}.
 The perturbed diffeomorphism is $C^1$ close to the original one, \cite[\S 3 and \S 4]{RY}. After that they embed $\Lambda$ in a $C^1$ diffeomorphism $F$ defined on a surface \cite[\S 5]{RY}. Although this construction is made to embed the horseshoe on a $C^1$-Anosov diffeomorphism, the same can be done for any diffeomorphism.

\begin{Obs}
Perhaps it is worthwhile to note that it is crucial that we are working in the $C^1$-topology. Bowen, \cite{Bo1}, proved that $C^2$ diffeomorphisms
have no horseshoes with positive volume.
\end{Obs}

\subsection{Proof of Theorem \ref{Teo D}.}
We now make use of the construction in \cite{RY} to prove that arbitrarily near a diffeomorphism exhibiting a homoclinic tangency there is one which is not measure-expansive.
\medbreak

We start establishing some auxiliary lemmas proved elsewhere.

\begin{Lem} \label{flat-tangency}
Given a $C^1$ diffeomorphism $f:M\to M$ with a homoclinic tangency associated to a hyperbolic periodic point $p$ there is a $C^1$ near diffeomrphism $f_1$ presenting a flat homoclinic tangency, i. e., there is a small arc $J$ contained in $W^s(p,f_1)\cap W^u(p,f_1)$.
\end{Lem}
\begin{proof}
See \cite[Proposition 2.6]{PaVi2}.
\end{proof}

\begin{Lem} \label{small-horseshoes}
Given a $C^1$ diffeomorphism $f_1:M\to M$ with a flat homoclinic tangency associated to a hyperbolic periodic point $p$ there is a $C^1$ near diffeomorphism $f_2$ presenting a sequence of horseshoes $\widehat \Lambda_n$ such that for all $k\in\Z$:  $\diam(f^k(\widehat \Lambda_n)<r_n$ with $r_n\to 0$ when $n\to\infty$.
\end{Lem}
\begin{proof}
The proof is essentially the same as that of \cite[Subsection 2.2]{PaVi2}.
\end{proof}

\begin{Prop} \label{horseshoe-gordo}
Let $f_2:M\to M$ as in the thesis of Lemma \ref{small-horseshoes}.
There is a $C^1$-diffeomorphism $F:M\to M$ arbitrarily near $f_2$ presenting a sequence of horseshoes $\Lambda_n$ such that the Lebesgue measure $\mu(\Lambda_n)>0$ and $\diam(\Lambda_n)<2r_n$, where $r_n$ is as in Lemma \ref{small-horseshoes}.
\end{Prop}
\begin{proof}
We profit from the construction made in \cite{RY}. In fact we do not need to take care for the perturbed diffeomorphism to be Anosov, as is the case in \cite{RY}. Hence, in our case, to fit the construction in the global picture of the perturbations is easier than at \cite[\S 5]{RY}. 
Since the support of the perturbation needed to fatten the horseshoe $\widehat\Lambda_n$ is contained in a box $B_n\supset \widehat\Lambda_n$ such that $\lim_{n\to\infty}\diam(B_n)= 0$ (see \cite[\S 3]{RY}), it can be taken disjoint from the support of the previous perturbations needed to fatten $\widehat \Lambda_{j}$ for $j=1,\ldots, n-1$ (see \cite[\S 2 and \S 4]{RY}). From this it follows that $F$ is $C^1$- close to $f_2$ and has the desired sequence of horseshoes $\Lambda_n$ with positive Lebesgue measure.

Moreover, the construction of $\Lambda_n$ gives that the diameter of $\Lambda_n$ is about the same of that of $\widehat\Lambda_n$, so that we can assure that $\diam(\Lambda_n)<2r_n$ from $\diam(\widehat\Lambda_n)<r_n$.
\end{proof}

%
%

As a consequence we have
\begin{Teo}\label{ultimo}
Let $M$ be a smooth compact surface. Given a $C^1$-diffeomorphism $f:M\to M$ exhibiting a homoclinic tangency associated to a hyperbolic periodic point $p$, it is $C^1$-approximated by a diffeomorphism $F:M\to M$ such that $F$ is not measure expansive with respect to any absolutely continuous invariant measure respect to Lebesgue.
\end{Teo}
\begin{proof}
Let $F:M\to M$ be the $C^1$ diffeomorphism constructed as in Proposition \ref{horseshoe-gordo} above. Then for every horseshoe $\Lambda_n$ associated to $F$ there is a hyperbolic periodic point $p_n\in \Lambda_n$ such that
$\mu(\Gamma_{2r_n}(p_n)\geq \mu(\Lambda_n))>0$ where $\mu\ll Leb$ and $f^*\mu=\mu$\,. Since $r_n\to 0$ when $n\to \infty$ the proof follows.
\end{proof}

Theorem \ref{ultimo} gives the proof of Theorem \ref{Teo D}.

\subsection*{Acknowledgements.}
J. Vieitez thanks IMPA and {\it Instituto de Matematica, Universidade Federal do Rio de Janeiro} for their kind hospitality during the preparation of this paper.

\begin{tabbing}
Universidade Federal do Rio de Janeiro, \hspace{1cm}\= Facultad de Ingenieria, \kill
 M. J. Pacifico, \> J. L. Vieitez, \\
Instituto de Matematica, \> Instituto de Matematica,\\
Universidade Federal do Rio de Janeiro, \> Facultad de Ingenieria,\\
C. P. 68.530, CEP 21.945-970, \> Universidad de la Republica,\\
Rio de Janeiro, R. J. , Brazil. \> CC30, CP 11300,\\
                                \> Montevideo, Uruguay\\
{\it pacifico@im.ufrj.br} \> {\it jvieitez@fing.edu.uy}
\end{tabbing}

\end{document}